\documentclass{article}[11pt]

\usepackage{graphicx,tikz,color} 
\usepackage{amsmath,amsfonts,amssymb,amsthm,latexsym}

\newtheorem{theorem}{Theorem}[section]
\newtheorem{definition}[theorem]{Definition}
\newtheorem{proposition}[theorem]{Proposition}

\newtheorem{conjecture}[theorem]{Conjecture}
\newtheorem{corollary}[theorem]{Corollary}
\newtheorem{remark}[theorem]{Remark}
\newtheorem{lemma}[theorem]{Lemma}
\newtheorem{claim}{Claim}[section]

\newenvironment{claimproof}[1]{{\it\noindent{Proof.}}\space#1}{\footnotesize \hfill \ensuremath{(\square)} \medskip}

\newcommand{\ggrt}{\gamma_{gr}^t}
\newcommand{\gt}{\gamma_{t}}

\begin{document}

\title{On graphs with equal total domination and Grundy total domination number}

\author{
Tanja Gologranc$^{a,b}$ \and Marko Jakovac$^{a,b}$ \and Tim Kos$^{b}$  \and Tilen Marc$^{b, c}$}

\maketitle

\begin{center}
$^a$ Faculty of Natural Sciences and Mathematics, University of Maribor, Slovenia\\
\medskip

$^b$ Institute of Mathematics, Physics and Mechanics, Ljubljana, Slovenia\\
\medskip

$^c$ Faculty of Mathematics and Physics, Ljubljana, Slovenia\\
\medskip

\end{center}

\begin{abstract}
A sequence $(v_1,\ldots ,v_k)$ of vertices in a graph $G$ without isolated vertices is called a total dominating sequence if every vertex $v_i$ in the sequence totally dominates at least one vertex that was not totally dominated by $\{v_1,\ldots , v_{i-1}\}$ and $\{v_1,\ldots ,v_k\}$ is a total dominating set of $G$. The length of a shortest such sequence is the total domination number of G ($\gt(G)$), while the length of a longest such sequence is the Grundy total domination number of $G$ ($\ggrt(G)$). In this paper we study graphs with equal total and Grundy total domination number. We characterize bipartite graphs with both total and Grundy total domination number equal to 4, and show that there is no connected chordal graph $G$ with $\gt(G)=\ggrt(G)=4$. The main result of the paper is a characterization of regular bipartite graphs with $\gt(G)=\ggrt(G)=6$ proved by establishing a surprising correspondence between existence of such graphs and a classical but still open problem of the existence of certain finite projective planes.
\end{abstract}


\noindent
{\bf AMS Subj.\ Class.\ (2010)}: 05C69, 05B15

\section{Introduction}

The total domination was introduced in 1980~\cite{cdh80}, and has been extensively studied since. The interest in this combinatorial property is motivated by it simplicity, natural applications, and results connecting it to many other combinatorial parameters, see survey monograph~\cite{HY}. A set $S$ of vertices of a graph $G=(V,E)$ is a {\it total dominating set}, if every vertex of $G$ has a neighbor in $S$. The cardinality of a minimum total dominating set in $G$ is called the {\it total domination number} of $G$ and is denoted by $\gt(G)$.

In~\cite{bhr2016}, an invariant that strives for the biggest total dominating set of a graph, was introduced. Let $G$ be a graph without isolated vertices and denote by $N(v)$ the \emph{(open) neighborhood} of a vertex $v$, i.e.\ the set of all the neighbors of $v$. Call the sequence $S=(v_1,\ldots , v_k)$ of distinct vertices of $G$ a \emph{legal} sequence, if for any $i \in \{2,\ldots , k\}$, vertex $v_i$ totally dominates at least one vertex from $V(G) \setminus \bigcup_{j<i}N(v_j)$, i.e.\
$$N(v_i) \setminus \bigcup_{j<i} N(v_j) \neq \emptyset.$$
Given a sequence $S=(v_1,\ldots,v_k)$ of $G$, denote  by  $\hat{S}$ the corresponding set of vertices $\{v_1,\ldots , v_k\}$. If $S$ is a legal sequence and $\hat{S}$ is a total dominating set of $G$, then $S$ is a {\it total dominating sequence} of $G$. The maximum length of a total dominating sequence in $G$ is called the {\it Grundy total domination number} of $G$ and it is denoted by $\ggrt(G)$. The corresponding sequence is a {\it Grundy total dominating sequence} of $G$.

This recently introduced parameter has received many attention after its introduction followed up by many interesting results. Similarly as the decision version of the total domination problem asking if $\gt(G)$ of a graph $G$ is smaller than some constant, also the decision version of the Grundy total domination number is NP-complete. In fact the problem is already hard in bipartite graphs~\cite{bhr2016} and also in split graphs~\cite{BKN}. On the other hand, efficient algorithms for computing the Grundy total domination number are known for trees, bipartite distance-hereditary graphs, and $P_4$-tidy graphs~\cite{BKN}.

Many bounds for the Grundy total domination number are known for various families of graphs, such as connected regular graphs and graph products~\cite{BBG-2017, bhr2016}. A simplest upper bound for a general graph $G$ is $\ggrt(G)=|V(G)|$ and graphs obtaining this bound were characterized in~\cite{bhr2016}. On the other hand, a natural lower bound for $\ggrt(G)$ is the total domination number of $G$. In this paper we focus on the extremal graphs obtaining this bound, continuing the work from~\cite{bhr2016} where it was proved that $\gt(G)=\ggrt(G)=2$ holds exactly for the complete multipartite graphs and that there is no graph with $\gt(G)=\ggrt(G)=3$. We focus on the sequential cases, showing that the case $\gt(G)=\ggrt(G)=4$ leads to simple extremal graphs, while the case $\gt(G)=\ggrt(G)=6$ is far more complicated with a strong connection with projective planes, latin squares, etc.

A game version of total domination number $\gamma_{tg}$ was defined in \cite{HKR} and it follows from its definition that  $\gt(G)\leq \gamma_{tg}(G) \leq \ggrt(G)$. Since in this paper we work with graphs for which the latter is an equality, we are also dealing with extremal graphs for the game total domination number. Finding such graphs is an open problem which is already interesting when restricted to special graph classes. For the class of trees the problem was solved in~\cite{HR}. Similar extremal problems are investigated also for game version of the domination number $\gamma_g$, where the trees with the same domination and game domination number were characterized in~\cite{game-version}.

The paper is organized as follows. In Section~\ref{sec:prelim} we present relevant results about multigraphs setting up the stage for Grundy total domination sequences in bipartite graphs. We continue in Section~\ref{sec:4} with a characterization of bipartite graphs with both total domination number and Grundy total domination number 4. We also prove that there is no chordal graph with  both total domination number and Grundy total domination number being equal to 4. Finally in Section~\ref{sec:regBib6} we show that a classification of graphs $G$ with $\gt(G)=\ggrt(G)=k \geq 6$ is a much harder problem connected to other classical open problems. We characterize regular bipartite graphs with $\gt(G)=\ggrt(G)= 6$ and prove that the existence of those graphs is closely related to the existence of finite projective plains or equivalently to the existence of perfect family of pairwise orthogonal Latin squares.

\section{Bipartite graphs as multigraphs} \label{sec:prelim}

In this section we explain the connections between dominating sequences in bipartite graphs and similar concepts in hypergraphs. Most of the section is dedicated to presenting the hypergraph terminology and results from~\cite{bhr2016} that have a direct corollary on Grundy total domination in bipartite graph. The result is asserted in Corollary \ref{col:bipartiteEqual}, where a reader wanting to avoid technical details can skip to. Note that an important implication of this section is the nonexistence of bipartite graphs with odd and equal total and Grundy total domination numbers.

Let ${\mathcal{H}} = (X, E)$ be a hypergraph with no isolated vertices. An {\it edge cover} of $\mathcal{H}$ is a set of hyperedges from $E$ that cover all vertices of $X$. The {\it covering number} of $\mathcal{H}$, $\rho({\mathcal{H}})$, is the minimum number of hyperedges in an edge cover of ${\mathcal{H}}$. A {\it legal (hyperedge) sequence} of ${\mathcal{H}}$, ${\mathcal{C}}=(C_1,\ldots , C_k)$, is a sequence of hyperedges from $E$ such that $C_i \setminus \bigcup_{j < i}C_j \neq \emptyset $ for any $i \in \{1,\ldots , k\}$. If ${\mathcal{C}}=(C_1,\ldots , C_k)$ is a legal sequence and $\{C_1,\ldots , C_k\}$ is an edge cover of ${\mathcal{H}}$, then ${\mathcal{C}}$ is an {\it edge covering sequence} of ${\mathcal{H}}$. The maximum length $k$ of an edge covering sequence of ${\mathcal{H}}$ is called the {\it Grundy covering number} of ${\mathcal{H}}$, $\rho_{\textrm{gr}}({\mathcal{H}})$. A {\it legal transversal sequence} is a sequence $S=(v_1,\ldots , v_t)$ of vertices from $X$ such that for each $i$ there exists an edge $E_i \in E$ such that $v_i \in E_i$ and $v_j \notin E_i$ for all $j < i$. The longest possible legal transversal sequence in ${\mathcal{H}}$ is {\it Grundy transversal sequence} and its length is the {\it Grundy transversal number} of ${\mathcal{H}}$, $\tau_{\textrm{gr}}({\mathcal{H}})$.

The {\it incidence graph} of a hypergraph ${\mathcal{H}}=(X,E)$ is the bipartite graph $G=(V,E)$, whose vertex set can be partitioned into independent sets $\tilde{X}$ and $\tilde{E}$ that correspond to the set of vertices $X$ and hyperedges, respectively. A vertex $\tilde{x} \in \tilde{X}$ is adjacent to $\tilde{E_1} \in \tilde{E}$ if and only if $x \in E_1$. It follows from definitions that the Grundy covering number of a hypergraph ${\mathcal{H}}$ coincides with the maximum length of a legal sequence in $\tilde{E}$ that totally dominates $\tilde{X}$ in the incidence graph of ${\mathcal{H}}$. On the other hand, it was proved in~\cite{bhr2016} that the Grundy transversal number of a hypergraph ${\mathcal{H}}$ coincides with the maximum length of a legal sequence in $\tilde{X}$ that totally dominates $\tilde{E}$ in the incidence graph of ${\mathcal{H}}$. This means that the Grundy total domination number of the incidence graph of ${\mathcal{H}}$ coincides with $\tau_{\textrm{gr}}({\mathcal{H}})+\rho_{\textrm{gr}}({\mathcal{H}})$. Even more, in~\cite{bhr2016} the following results were proved.
\begin{proposition}\cite[Proposition 8.3]{bhr2016}
The Grundy transversal number of an arbitrary hypergraph ${\mathcal{H}}$ equals the Grundy covering number of ${\mathcal{H}}$.
\end{proposition}
\begin{theorem}\cite[Theorem 8.4]{bhr2016}
If ${\mathcal{H}}$ is a hypergraph and $G$ the incidence graph of ${\mathcal{H}}$, then $\ggrt(G)=2\rho_{\textrm{gr}}({\mathcal{H}}).$
\end{theorem}

Let $G=A \cup B$ be a bipartite graph and ${\mathcal{H}}=(V(G),{\mathcal{N}}(G))$ the open neighborhood hypergraph of $G$. Then ${\mathcal{H}}$ has two connected components ${\mathcal{H}}_1$ and ${\mathcal{H}}_2$ and the incidence graph of ${\mathcal{H}}_i$ is isomorphic to $G$ for every $i \in \{1,2\}$.   

\begin{corollary}\label{col:bipartiteEqual}
Let $G$ be a bipartite graph with bipartition $A \cup B$. Then the Grundy total domination number of $G$ is even and for any Grundy total dominating sequence $S=(v_1,\ldots , v_{2k})$ it follows that $|A \cap \hat{S}|=|B \cap \hat{S}|=k$.
\end{corollary}

\section{Graphs with $\gt(G)=\ggrt(G)=4$}\label{sec:4}

In this section we characterize bipartite graphs with $\gt(G)=\ggrt(G)=4$. First define tow distinct vertices $u$ and $v$ of a graph $G$ to be {\it false twins} if $N(u)=N(v)$. A graph is {\it false twin-free} (also known as \textit{thin}) if it has no false twins. Now
notice that if $G$ is a graph and a vertex is added and connected to the neighborhood of an arbitrary vertex of $G$, then the total domination number and the Grundy total domination number does not change. In other words, the question of characterizing extremal graphs is only interesting for false twin-free graphs.

\begin{theorem}\label{th:regularBipartite4}
Let $G$ be a bipartite false twin-free graph. Then $\gt(G)=\ggrt(G)=4$
if and only if $G$ is isomorphic to the graph $K_{n,n}-M$, $n \geq 2$, where $M$
denotes an arbitrary perfect matching of $K_{n,n}$.
\end{theorem}

\begin{proof}
First, let $G$ be a graph isomorphic to $K_{n,n}- M$, where $M$ is a perfect matching of $K_{n,n}$ and $n \geq 2$. Let $A, B$ be the bipartition of $G$. We need at least two vertices from $A$ to totally dominate $B$ and at least two vertices from $B$ to totally dominatea $A$. Since any two vertices of $A$ totally dominate $B$ and any two vertices of $B$ totally dominate $A$, $\gt(G)=\ggrt(G)=4$.

For the converse suppose that $\gt(G)=\ggrt(G)=4$.   
Let again $A, B$ be the bipartition of a graph $G$ with $|A|=m$ and $|B|=n$ and let $D$ be a Grundy total dominating sequence such that $\hat{D}$ is a minimum total dominating set. Then it follows from Corollary~\ref{col:bipartiteEqual} that $|A \cap \hat{D}| =|B \cap \hat{D}|=2$.

Denote with $a_1,a_2$ and $b_1,b_2$ the vertices in $A \cap \hat{D}$ and $B \cap \hat{D}$, respectively.
From these conditions it is clear that $m,n \geq 2$. Since $\hat{D}$ is a minimum total
dominating set of $G$, we have $N(a) \neq B$ and $N(b) \neq A$ for every vertex
$a \in A$ and $b \in B$. Hence, $|N(a)| \leq n-1$ and $|N(b)| \leq m-1$ for every
vertex $a \in A$ and $b \in B$. Suppose that there exists a vertex $a \in A$ such
that $|N(a)| \leq n-2$. Then there exist two vertices $x,y \in B$ which are not
adjacent to $a$. By assumption there are no false twins in $G$, and hence without loss of generality
we can assume that $y$ has a neighbor not adjacent to $x$. Thus, $(x,y, a_1,a_2)$ is a legal sequence that does not totally
dominate $G$, a contradiction with $\ggrt(G)=4$. This gives $|N(a)| = n-1$ for every vertex $a \in A$. By symmetry,
we also have $|N(b)| = m-1$ for every vertex $b \in B$. Therefore, the number of
edges in $G$ equals $|E(G)|=m(n-1)=n(m-1)$. From this equation we get $m=n$.

Summing all things up, $G$ must be a bipartite graph on $2n$ vertices,
and every vertex in $G$ has degree $n-1$. Thus, $G$ is isomorphic to the
graph $K_{n,n}-M$, $n \geq 2$, where $M$ can be an arbitrary perfect matching
of $K_{n,n}$. 
\end{proof}

The above theorem motivates the question of the existence of non-bipartite graphs with $\gt(G)=\ggrt(G)=4$. It is easy to construct such disconnected graphs, as $G$ can be a graph with two connected components $G_1,G_2$, where each component induces a graph with $\gt(G_i)=\ggrt(G_i)=2$. Thus those graphs are exactly graphs with two connected components, where each component is a complete multipartite graph~\cite{bhr2016}. As we are focused just on false twin-free graphs, those graphs restrict to graphs with two connected components, where each component is a complete graph. Since those are trivial cases obtained from $\gt(G)=\ggrt(G)=2$, which are not really interesting and since there is no graph $G$ with $\gt(G)=1$, all the remaining graphs with $\gt(G)=\ggrt(G)=4$ are connected. We performed a computer check showing that there are no such graphs on up to 20 vertices. We strongly believe that there are in fact none.

\begin{conjecture}
Let $G$ be a connected false twin-free graph. Then $\gt(G)=\ggrt(G)=4$ if and only if
$G$ is isomorphic to the graph $K_{n,n}-M$, $n \geq 3$, where $M$
denotes an arbitrary perfect matching of $K_{n,n}$.
\end{conjecture}

We continue the section with the proof of the correctness of the conjecture in the class of chordal graphs.
Recall that a graph is {\it chordal} if it contains no induced cycles of length greater than 3. A vertex $v$ of a graph $G$ is called {\it simplicial} if the subgraph of $G$ induced by $N[v]$ is a complete graph. Every chordal graph has at least one simplicial vertex~\cite{Di}.

Let $S=\{v_1,\ldots , v_k\}$ be a total dominating set of $G$. We call the set $N(v_i)\setminus \bigcup_{j\neq i} N(v_j)$ the {\it private neighborhood} of $v_i$.

\begin{theorem}
There is no connected chordal graph $G$ with $\gt(G)=\ggrt(G)=4$. 
\end{theorem}
\begin{proof}
Suppose that there exists a connected chordal graph $G'$ with $\gt(G)=\ggrt(G)=4$ and from all the graphs having these properties choose $G$ with the smallest number of vertices. Let $x$ be a simplicial vertex of $G$ and let $H=G\setminus \{x\}$. Since $\gt(H) \leq \ggrt(H)$ and $G$ has the smallest order with $\gt(G) =\ggrt(G) =4$, $\gt(H) \leq 3$. As total dominating set of $H$ together with a vertex from $N(x)$ is a total dominating set of $G$, $\gt(H)=3$. Let $S=\{u_1,u_2,u_3\}$ be a minimum total dominating set of $H$. It is clear that the subgraph of $H$ induced by $S$ is connected.

\begin{claim}\label{c:notN(x)}
Let $S'$ be an arbitrary minimum total dominating set of $H$. Then $S' \cap N(x) = \emptyset$.
\end{claim}
\begin{claimproof}
Let $S'=\{a,b,c\}$ and suppose that $a \in S' \cap N(x)$. Then $S'$ is a total dominating set of $G$, a contradiction.
\end{claimproof}

Suppose first that there exists $i\in \{1,2,3\}$ such that $u_i$ has no private neighbors in $H\setminus N(x)$.
Since $S$ is the smallest total dominating set of $H$, any vertex from $S$ has at least one private neighbor. Therefore all private neighbors of $u_i$ are in $N(x)$. If $N(x)$ is the private neighborhood of $u_i$, then $(N(u_j) \cup N(u_k)) \cap N(x) = \emptyset$, where $\{i,j,k\} = \{1,2,3\}$. Then $(x,u_i,u_j,u_k,x')$, where $x'$ is an arbitrary vertex from $N(x)$, is a legal total dominating sequence of $G$ (since graph induced by $S$ is connected, $u_i$ has a neighbor in $H\setminus N(x)$), which is a contradiction as $\ggrt(G)=4$.  Therefore there exists $x' \in N(x)$ that is not in the private neighborhood of $u_i$. Hence $\{x',u_j,u_k\}$ is total dominating set of $H$, a contradiction with Claim~\ref{c:notN(x)}.

We have proved that $u_i$ has some private neighbors in $H\setminus N(x)$ for any $i \in \{1,2,3\}$. Let $x' \in N(x)$. Then $(x,u_1,u_2,u_3,x')$ is a legal total dominating sequence of $G$ of length 5, the final contradiction. 
\end{proof}

\section{Graphs with $\gt(G)=\ggrt(G)=6$}
\label{sec:regBib6}

In the previous section we have seen that it is possible to classify the extremal bipartite graphs with $\gt(G)=\ggrt(G)=4$. The purpose of this section is to show that for higher values the situation is much more complicated. In fact, we shall prove that the existence is closely connected with the existence of finite affine planes, one of the oldest and still not solved combinatorial questions.

We begin the section with basic concepts about projective planes, affine planes and Latin squares. For notation and terminology we follow~\cite{stinson}.

A {\it Latin square} of order $n$ with entries from an $n$-set $X$ is an $n \times n$
array $L$ in which every cell contains an element of $X$ such that every row of $L$ is a
permutation of $X$ and every column of $L$ is a permutation of $X$. Let $L_1$ and $L_2$ be Latin squares of order $n$ with entries from $X$ and $Y$, respectively. We say that $L_1$ and $L_2$ are
{\it orthogonal} Latin squares provided that, for every $x \in X$ and for every $y \in Y$,
there is a unique cell $(i,j)$ such that $L_1(i,j) = x$ and $L_2(i,j)=y$. We say that Latin squares $L_1,\ldots , L_s$ of order $n$ are {\it mutually orthogonal}, if $L_i$ and $L_j$ are orthogonal for any $1 \leq i < j \leq s$. A set of mutually orthogonal Latin squares of order $n$ will be denoted by MOLS($n$). It is easy to see that there are at most $n-1$ mutually orthogonal Latins squares of order $n$. If there exist $n-1$ MOLS($n$) $L_1,\ldots , L_{n-1}$ we say that $\{ L_1,\ldots , L_{n-1}\}$ is a perfect orthogonal family of Latin squares.

A {\it{design}} is a pair $(X,{\mathcal{A}})$ such that $X$ is a set of elements called points, and
${\mathcal{A}}$ is a collection of nonempty subsets of $X$ called blocks.
Let $v,k,\lambda$ be positive integers such that $v > k \geq 2$. A $(v,k,\lambda)$-balanced incomplete block design (abbreviated $(v,k,\lambda)$-BIBD) is a
design $(X,{\mathcal{A}})$ such that $|X| = v$, each block contains exactly $k$ points, and every pair of distinct points is contained in exactly $\lambda$ blocks. An $(n^2 + n + 1, n + 1, 1)$-BIBD with $n \geq 2$ is called a projective plane of order $n$. An BIBD design where $|X|=n^2$, the number of blocks equals $n^2 + n$, each block contains $n$ points, each point is contained in exactly $n+1$ blocks and every pair of distinct points is contained in exactly $1$ block is called an {\it affine
plane} of order $n$.

What we need in the proof of our main result is the following.
\begin{definition}\cite[Definition 6.36]{stinson}
Let $s \geq 2$ and $q \geq 1$ be integers. An \emph{orthogonal array} $OA(s, q)$
is an $q^2 \times s$ array $A$, with entries from a set $X$ of cardinality $q$ such that, within
any two columns of $A$, every ordered pair of symbols from $X$ occurs in exactly one
row of $A$.
\end{definition}

$OA(s, q)$ can be seen as a collection of $q^2$ words of length $s$ over an alphabet of $q$ letters, such that each pair of words coincide in at most one place. The following is immediate from the definition of an orthogonal array but we point it out since it will be used latter:

\begin{lemma}\label{columns}
Let $s, q \in \mathbb{N}$ be chosen such that there exist  an orthogonal array $OA(s,q)$ with entries from $\{1, \ldots ,q\}$. Then every column contains exactly $q$ elements $i$ for any $i \in \{1, \ldots, q\}.$ 
\end{lemma}

Orthogonal arrays are connected with MOLS in the following way:

\begin{theorem}\cite[Theorem 6.38]{stinson}\label{thm:array_to_mols}
Suppose that $s \geq 3$ and $q \geq 1$ are integers. Then $s - 2$ MOLS($q$)
exist if and only if an $OA(s,q)$ exists.
\end{theorem}

Notice that the case $OA(q + 1, q)$ is extremal in the sense that $s$ cannot be greater. To see this just consider the first row $x$ of a $OA(s, q)$ and count how many rows have the same latter at fixed place as $x$. By Lemma \ref{columns} there are $q-1$ rows that have the same first latter, $q-1$ rows that have the same second latter, etc. By definition of an orthogonal array all this rows must be different, hence there are $s(q-1)$ of them. Since there are $q^2$ rows, $s(q-1) + 1 \leq q^2$. Hence $s\leq q+1$ and in the extremal case $OA(q + 1, q)$ we must have a collection of $q^2$ words of length $q+1$ over an alphabet of $q$ letters, such that each pair of words coincide in \emph{exactly} one place. 

This extremal case is extremly important. By above, it is equivalent to an existence of $q-1$ MOLS($q$).
Furthermore, MOLS are connected with other classical constructions:

\begin{theorem}\cite[Theorem 6.32]{stinson}\label{thm:equivalence}
Let $q \geq 2$. Then the existence of any one of the following designs
implies the existence of the other two designs:
\begin{enumerate}
\item $q-1$ MOLS($q$).
\item A finite affine plane of order $q$.
\item A projective plane of order $q$.
\end{enumerate}
\end{theorem}

Since it is known that for every prime power $q \geq 2$, there exists a projective plane of order $q$~\cite{stinson}, we also know that there exists a perfect orthogonal family of Latin squares of order $q$. Then Theorem~\ref{thm:array_to_mols} implies the existence of orthogonal array $OA(q+1,q)$.

%

We are ready for the characterization of regular bipartite graphs with $\gt(G)=\ggrt(G)=6$. We start with a simple lemma studying the neighborhoods of such graphs.

\begin{lemma}\label{l:-1}
Let $G$ be a bipartite false twin-free graph having $\gt(G)=\ggrt(G)=6$. Let $A \cup B$ be a bipartition of a graph $G$. Then $|N(\{a_1,a_2\})|=|B|-1$ for any $a_1\neq a_2 \in A$ and  $|N(\{b_1,b_2\})|=|A|-1$ for any $b_1\neq b_2 \in B$.
\end{lemma}
\begin{proof}
Let $a_1\neq a_2$ be arbitrary vertices from $A$, let $B'=N(\{a_1,a_2\})$ and let $A'$ be the set of all vertices from $A$ whose neighborhoods are contained in $B'$, i.e.\ $N(a') \subseteq B'$ for $a' \in A'$. It follows from Corollary~\ref{col:bipartiteEqual} that any Grundy total dominating sequence contains exactly three vertices $x_1,x_2,x_3$ from $A$ and exactly three vertices $y_1,y_2,y_3$ from $B$. Since $\gt(G)=6$, $\{a_1,a_2,y_1,y_2,y_3\}$ is not a total dominating set of $G$. Therefore $B-B' \neq \emptyset$. Since any Grundy dominating sequence contains three vertices from $A$, the set $\{a_1,a_2,a\}$ totally dominates $B$ for any $a \in A\setminus A'$. Therefore $B-B' \subseteq N(a)$ for any $a \in A \setminus A'$ and hence $N(b)=A\setminus A'$ for any $b \in B\setminus B'$. This implies that all vertices from $B\setminus B'$ have the same open neighborhoods and hence they are falls twins. As $G$ is false twin-free, $|B \setminus B'|=1$. The proof of $|N(\{b_1,b_2\})|=|A|-1$ for any $b_1\neq b_2 \in B$ goes in the same way.
\end{proof}

\begin{remark}
Let $G$ be a regular, bipartite graph with bipartition $A \cup B$. Then $|A|=|B|$.
\end{remark}

\begin{theorem}\label{th:regularBipartite6}
Let $n, k \in \mathbb{N}$. Then an $(n-k)$-regular bipartite false twin-free graph $G$ on $2n$ vertices with $\gt(G)=\ggrt(G)=6$ exists if and only if $n=k^2-k+1$ and there exists a finite projective plane of order $k-1$ (equivalently an affine plane of order $k-1$, or $k-2$ MOLS$(k-1)$, or $OA(k,k-1)$).
\end{theorem}

\begin{proof}
Suppose that $G$ is $(n-k)$-regular bipartite graph $G$ on $2n$ vertices with $\gt(G)=\ggrt(G)=6$ and bipartiton $A, B$. Let $a_1\neq a_2$ be arbitrary vertices from $A$. Lemma~\ref{l:-1} implies that $|N(\{a_1,a_2\})|=n-1$. Let $B_2=N(a_1)-N(a_2)$, $B_1=N(a_2)-N(a_1)$, and $B'=N(a_1) \cap N(a_2)$. Since $G$ is $(n-k)$-regular and $|N(\{a_1,a_2\})|=n-1$, it follows that $|B'|=n-2k+1$ and $|B_1|=|B_2|=k-1.$ Let $A'$ be the set of all vertices from $A\setminus \{a_1,a_2\}$ whose neighborhoods are contained in $B_1 \cup B_2 \cup B'$ and let $A''=A \setminus (A' \cup \{a_1,a_2\})$. Denote $\ell=|A'|$, which implies that $|A''|=n-\ell-2$ and let $A'=\{a_3,\ldots, a_{\ell+2}\}.$ 

\begin{claim}\label{claim:k-1}
Let $x \in A'$. Then $B_1 \cup B_2 \subseteq N(x)$ and $x$ has exactly $k-1$ non-neighbors in $B'$. 
\end{claim}
\begin{claimproof}
Suppose that there exists $x \in A'$ such that $B_1 \cup B_2 \nsubseteq N(x)$. Without loss of generality we may assume that $x$ is not adjacent to $b_1 \in B_1$. Then $(a_1,x,a_2)$ is a legal dominating sequence (note that $N(a_1) \neq N(x)$, as $G$ is false twin-free) that does not totally dominate whole $B$, a contradiction with the fact that any Grundy total dominating sequence contains exactly three vertices in $A$. Therefore $B_1 \cup B_2 \subseteq N(x)$ for any $x \in A'$. Since $x$ has degree $n-k$, $x$ has $n-3k+2$ neighbors in $B'$. In other words, $x$ is nonadjacent to exactly $k-1$ vertices from $B'$.
\end{claimproof}

Any vertex $a \in A\setminus A''$ has exactly $k-1$ noneighbors in $B_1 \cup B_2 \cup B'$. Denote the non-neighbors of $a_i$ in $B_1 \cup B_2 \cup B'$ by $B_i$ for any $i \in \{3, \ldots ,\ell+2\}$. Note that this extends the definition for $i \in \{1,2\}$. Claim~\ref{claim:k-1} implies that $|B_i|=k-1$ for any $i \in \{1,\ldots , \ell+2\}$.

\begin{claim}\label{c:non-neighbors}
Any vertex $b \in B_1 \cup B_2 \cup B'$ has exactly one non-neighbor in $\{a_1,a_2\} \cup A'.$
\end{claim}
\begin{claimproof}
We already proved that this holds for all $b \in B_1 \cup B_2$, as $a_2$ is the only non-neighbor of $b \in B_2$ and $a_1$ is the only non-neighbor of $b \in B_1$. Since $B'=N(a_1) \cap N(a_2)$ non-neighbors of $b \in B'$ are from $A'$. Suppose first that $b \in B'$ is adjacent to all vertices in $\{a_1,a_2\} \cup A'.$ By Lemma \ref{l:-1}, $B \setminus (B_1 \cup B_2 \cup B')$ consist of one vertex, while all vertices in $A''$ are by definition adjacent to this vertex. Hence three vertices from $A$ together with $b$ and the vertex in $B \setminus (B_1 \cup B_2 \cup B')$ form a total dominating set of $G$, a contradiction. Suppose now that $b \in B'$ is not adjacent to two different vertices $a,a'\in A'.$ Then $(a,a',a_1)$ is a legal dominating sequence in $A$ that does not totally dominates whole $B$, a contradiction. Hence any vertex $b \in B_1 \cup B_2 \cup B'$ has exactly one non-neighbor in $\{a_1,a_2\} \cup A'.$
\end{claimproof}

\begin{claim}\label{c:partition}
$B_3,\ldots , B_{\ell+2}$ is a partition of $B'$.
\end{claim}
\begin{claimproof}
Let $b \in B'$. It follows from Claim~\ref{c:non-neighbors} that there exists $a_i \in A'$ such that $a_ib \notin E(G)$. Therefore $b \in B_i$. Suppose that $x \in B_i \cap B_j$, $i, j \in \{3,\ldots , \ell+2\}, i\neq j$. Then $x \in B'$ has at least two non-neighbors $a_i,a_j$ in $ A'$, which contradicts Claim~\ref{c:non-neighbors}.
\end{claimproof}

Claim~\ref{c:partition} implies the following equation. 
\begin{equation}\label{eq1}
|B'|=\ell(k-1)=n-2k+1.
\end{equation}

\begin{claim}\label{c:neighborsA''}
For any $a''\in A''$ and any $i \in \{1,\ldots , \ell+2\}$, $|N(a'')|\cap B_i=k-2$.
\end{claim}
\begin{claimproof}
 Let $a'' \in A''$. Since any two vertices from $A$ totally dominates $n-1$ vertices in $B$, $|N(\{a'',a_i\})|=n-1$ for any $i \in \{1,\ldots , \ell+2\}.$ This means that $a''$ is adjacent to all except one non-neighbor of $a_i$, i.e.\ $a''$ is adjacent to $k-2$ vertices from $B_i$ for any $i \in \{1,\ldots , \ell+2\}$. 
\end{claimproof}

\begin{claim}\label{c:sizes}
$|A|=n=k^2-k+1, |A'|=\ell=k-2, |A''|=(k-1)^2.$
\end{claim}
\begin{claimproof}
Let $a'' \in A''$. Since $a''$ is adjacent to the vertex from $B \setminus (B_1 \cup B_2 \cup B')$, Claim~\ref{c:neighborsA''} implies  that the degree of $a''$ is $(\ell+2)(k-2)+1$. Since $G$ is $(n-k)$-regular, $(\ell+2)(k-2)+1=n-k$. Combining this equation together with (\ref{eq1}), we obtain $|A|=n=k^2-k+1, |A'|=\ell=k-2$ and $|A''|=(k-1)^2.$
\end{claimproof}

Now we turn to the structure of the graph. It follows from the above, that if there exists a regular bipartite graph $G$ on $2(k^2-k+1)$ vertices with $\gt(G)=\ggrt(G)=6$, then for any $a'' \in A''$ we can choose for any $i \in \{1,\ldots , \ell+2\}$ exactly one vertex in $B_i$ that is not adjacent to $a''$ (Claim 4.4) such that any $b \in B_1 \cup B_2 \cup B'$ will be chosen (as a non-neighbor) $k-1$ times ($b$ has $k$ non-neighbors in $A$, one is from $\{a_1,a_2\} \cup A'$, hence $k-1$ non-neighbors are from $A''$) and each two vertices from $A''$ have exactly one common non-neighbor (Lemma~\ref{l:-1}). For each of the $(k-1)^2$ vertices in $A''$ we have to choose $k$ non-neighbors, one from each $B_i$ and any two vertices from $A''$ have exactly one common non-neighbor. Therefore, we can read from the graph $(k-1)^2$ words of length $k$ from alphabet $\{1,\ldots , k-1\}$, such that every two words coincide in exactly one place. In particular, the existence of the graph implies the existence of an orthogonal array $OA(k, k-1)$ which by Theorem \ref{thm:array_to_mols} is equivalent to $k -2$ MOLS($k-1$) (or finite projective plane of order $k-1$, or affine plane of order $k-1$ by Theorem \ref{thm:equivalence}). 

%

For the converse we will define a reverse construction of the above. Assume that $|A|=|B|=n=k^2-k+1$ and that there exists an orthogonal array $L\in OA(k,k-1)$ with entries from $\{1, \ldots ,k-1\}$. Then it follows from Lemma~\ref{columns} that  every column contains exactly $(k-1)$ elements $i$ for any $i \in \{1, \ldots, k-1\}.$
We will construct a bipartite regular graph $G$ with $\gt(G)=\ggrt(G)=6$. Let $A=\{a_1,\ldots , a_k, a_1',\ldots a_{(k-1)^2}'\}$. Let $B=B_1 \cup \ldots \cup B_k \cup \{b\}$, where $|B_i|=k-1$ for any $i \in \{1, \ldots , k\}$. Denote vertices of $B_i$ by $B_i=\{b_1^i,\ldots , b_{k-1}^i\}$.
Define edges of $G$ as follows. Connect any $a_i \in \{a_1,\ldots , a_k\}$ to all $b \in \bigcup_{\ell=1}^{i-1}B_{\ell} \cup \bigcup_{\ell=i+1}^k B_{\ell}$. Then connect any $a_i'$ to all $b_j \in B_{s}\setminus \{b_{L(i,s)}^s\}$ and to $b$, for any $s \in \{1,\ldots , k\}, i \in \{1,\ldots , (k-1)^2\}$.  We will first prove that $G$ is $(n-k)$-regular. The vertex $a_i\in \{a_1,\ldots , a_k\}$ is adjacent to all vertices in $B$ except to vertices from $B_i \cup \{b\}$. Hence the degree of $a_i$ is $n-k$. The vertex $a_i'\in \{a_1',\ldots , a_{(k-1)^2}'\}$ is adjacent to $k-2$ vertices from each $B_j$, $j\in \{1,\ldots , k\}$, and to $b$. Therefore the degree of $a_i'$ is $(k-2)k+1=k^2-2k+1=n-k$. The vertex $b$ is adjacent to all vertices from $\{a_1',\ldots , a_{(k-1)^2}'\}$ and consequently it has degree $k^2-2k+1=n-k$. Finally the vertex $b_i^j$, $i\in \{1,\ldots , k-1\}, j \in \{1,\ldots , k\}$, is adjacent to all vertices $a_{\ell}'\in \{a_1',\ldots , a_{(k-1)^2}'\}$ for which $L(\ell,j)\neq i$. Since it follows from Lemma~\ref{columns} that $L(\ell,j)=i$ for exactly $k-1$ indices $\ell \in \{1,\ldots , (k-1)^2\}$, $b_i^j$ has $(k-1)^2-(k-1)$ neighbors in $\{a_1',\ldots , a_{(k-1)^2}'\}$. The vertex $b_i^j$ is also adjacent to all vertices from $\{a_1,\ldots , a_k\}$ except $a_j$. Therefore the degree of $b_i^j$ is $(k-1)^2-(k-1)+(k-1)=k^2-2k+1=n-k$, which proves that $G$ is regular.  

Finally we will prove that $\gt(G)=\ggrt(G)=6$. Let $D$ be a minimum total dominating set of $G$ and let $S$ be a Grundy total dominating sequence of $G$. Note first that any two vertices from $\{a_1, \ldots a_k\}$ totally dominate all vertices from $B$ except $b$. Vertices $a_i\in \{a_1, \ldots , a_k\}$, $a_j' \in \{a_1',\ldots , a_{(k-1)^2}'\}$ totally dominate all vertices from $B$ except one vertex from $B_i$, that is $b_{L(j,i)}^i$. For any $j_1,j_2 \in \{1,\ldots , (k-1)^2\}$ there exists exactly one $i \in \{1,\ldots, k\}$ such that $L(j_1,i)=L(j_2,i)$. Therefore any two vertices $a_{j_1}',a_{j_2}' \in \{a_1',\ldots , a_{(k-1)^2}'\}$ dominates all vertices from $B$ except $b_{L(j_1,i)}^i$ for which $L(j_1,i)=L(j_2,i)$. Therefore any two vertices from $A$ totally dominate all except one vertex from $B$. Hence $|D \cap A|=3$ and $|\hat{S} \cap A|=3$. For any $i \in \{1,\ldots ,k \}$ the vertices $b$ and an arbitrary vertex $b' \in B_i$ totally dominate all vertices from $A$ except $a_i$. Since $b_{j_1}^i$ is adjacent to all vertices from $A$ except $a_i$ and those vertices $a_{ell}'$ from $\{a_1',\ldots , a_{(k-1)^2}'\}$ for which $L(\ell,i)=j_1$, two different vertices $b_{j_1}^i,b_{j_2}^i \in B_i$ totally dominate all vertices from $A$ except $a_i$. Finally let $b_{j_1}^i \in B_i$ and $b_{j_2}^j \in B_j$. Since $L$ is an orthogonal array, $L(\ell,i)=j_1, L(\ell,j)=j_2$ hold for exactly one $\ell \in \{1,\ldots , (k-1)^2\}$. Therefore vertices $b_{j_1}^i,b_{j_2}^j$ totally dominate all vertices from $A$ except $a_{\ell}'$ with $L(\ell,i)=j_1, L(\ell,j)=j_2$. Hence any two vertices from $B$ totally dominate all except one vertex from $A$ and hence $|D \cap B|=|\hat{S} \cap B|=3$. Therefore $\gt(G)=\ggrt(G)=6$.

\end{proof}

Notice that for each $k-1$ being a prime power $p^i$ there exist a construction of projective plane of order $k-1$. Moreover all known constructions have a prime power order.

For example, if $k=3$, there exists, up to isomorphism, precisely one regular, bipartite graph on $2n=14$ vertices with $\gt(G)=\ggrt(G)=6$, corresponding to unique projective plane of order 2. See Figure~\ref{14vertces}. The next example corresponding to a unique projective plane of order 3 has 26 vertices. We have verified by a computer check that up to 26 vertices there are no other bipartite graphs satisfying $\gt(G)=\ggrt(G)=6$, leading to a suspicion that all of them are regular.

\begin{figure}[ht!]
    \begin{center}
        \begin{tikzpicture}[scale=1.0,style=thick,x=1cm,y=1cm]
        \def\vr{2.5pt} 

        \foreach \x  in {0,1,2,3,4,5,6}
         \foreach \y  in {0,2}
        {
        \filldraw [fill=black, draw=black,thick] (\x,\y) circle (3pt);
        }
        \filldraw [fill=black, draw=black,thick] (2,0) circle (3pt);

    \foreach \x  in {0,1,2,3}
    {
    \draw (0,2) -- (\x,0);
    }
		\foreach \x  in {0,1,4,5}
    {
    \draw (1,2) -- (\x,0);
    }
		\foreach \x  in {2,3,4,5}
    {
    \draw (2,2) -- (\x,0);
    }
		\foreach \x  in {1,3,5,6}
    {
    \draw (3,2) -- (\x,0);
    }
		\foreach \x  in {1,2,4,6}
    {
    \draw (4,2) -- (\x,0);
    }
		\foreach \x  in {0,3,4,6}
    {
    \draw (5,2) -- (\x,0);
    }
		\foreach \x  in {0,2,5,6}
    {
    \draw (6,2) -- (\x,0);
    }
    
        \end{tikzpicture}
    \end{center}
    \caption{A regular bipartite graph with $\gt(G)=\ggrt(G)=6.$}
    \label{14vertces}
\end{figure}

\begin{conjecture}
If $G$ is a connected false twin-free graph with $\gt(G)=\ggrt(G)=6$, then $G$ is a  regular bipartite graph.  
\end{conjecture}

\section*{Acknowledgements}
The authors are grateful to Zsolt Tuza for several useful comments. The authors also acknowledges the financial support from the Slovenian Research Agency (research core funding No.\ P1--0297 and research project No.\  J1--9109).



\end{document}